\documentclass[smallextended,envcountsect]{svjour3} 
\smartqed

\usepackage{cite}
\usepackage{longtable}
\usepackage{graphicx}              
\usepackage{amsmath}               
\usepackage{amssymb}
\usepackage{amsfonts}              
\usepackage{appendix}
\usepackage{comment}
\usepackage{caption}
\usepackage{color}
\usepackage{url}

\usepackage{algorithm,algorithmic}

\usepackage{enumerate}
\usepackage{arydshln} 
\usepackage{cases}  
\usepackage{empheq}

\usepackage{setspace} 
\usepackage{subfigure}
\usepackage{footnote}
\usepackage{lscape} 
\usepackage{threeparttablex}

\allowdisplaybreaks[1]

\makeatletter 
\newcommand*{\rom}[1]{\expandafter\@slowromancap\romannumeral #1@}
\makeatother

\newcommand{\J}{J}

\DeclarePairedDelimiterX\Set[2]{\lbrace}{\rbrace}%
 { #1 \,\delimsize| \,\mathopen{} #2 }

\newtheorem{instance}{Numerical Instance}
\newtheorem{assumption}{Assumption}

\newcommand{\mJ}{\mathcal{J}}
\newcommand{\mK}{\mathcal{K}}
\newcommand{\mS}{\mathcal{S}}

\makeatletter

\newcommand{\Rmnum}[1]{\expandafter\@slowromancap\romannumeral #1@}
\makeatother

\newlength\myindent
\newcommand\bindent{%
  \begingroup
  \setlength{\itemindent}{\myindent}
  \addtolength{\algorithmicindent}{\myindent}
}
\newcommand\eindent{\endgroup}

\begin{document}

\title{A Geometric Branch and Bound Method for a Class of Robust Maximization Problems of Convex Functions}

\author{Fengqiao Luo  \and  Sanjay Mehrotra }

\institute{Fengqiao Luo \at
              Department of Industrial Engineering and Management Science, Northwestern University \\
              Evanston, Illinois\\
              fengqiaoluo2014@u.northwestern.edu  
           \and
              Sanjay Mehrotra,  Corresponding author  \at
              Department of Industrial Engineering and Management Science, Northwestern University \\
              Evanston, Illinois\\
              mehrotra@northwestern.edu
}

\date{Received: date / Accepted: date}
\maketitle

\begin{abstract}
We investigate robust optimization problems defined for maximizing convex functions. For finite uncertainty set, we develop a geometric branch-and-bound algorithmic approach to solve this problem. The geometric branch-and-bound algorithm performs sequential piecewise-linear approximations of the convex objective, and solves linear programs to determine lower and upper bounds of nodes specified by the active linear pieces. Finite convergence of the algorithm to an $\epsilon-$optimal solution is proved. Numerical results are used to discuss the performance of the developed algorithm. The algorithm developed in this paper can be used as an oracle in the cutting surface method for solving robust optimization problems with compact ambiguity sets.
\end{abstract}

\section{Introduction}
We consider robust maximization problem with finitely-many candidate
objective functions: 
\begin{equation}\label{opt:FRO-F}
\underset{x\in X}{\textrm{max}}\textrm{ } \underset{k\in[K]}{\textrm{min}} \textrm{ } f_k(x), \tag{RM}
\end{equation}
where $X\subset \mathbb{R}^n$ is the feasible set of the decision variables $x$, 
$\{f_k\}^K_{k=1}$ is the set of $K$ candidate functions. 
The goal of the decision maker is to find a risk-averse optimal solution $x$ with respect 
to any choice of function index $k$ of the candidate function.
Note that \eqref{opt:FRO-F} may represent a pessimistic
view of the uncertainty. In this view, it is believed that the nature intends to choose a functional form which is against
the decision. In contrast, an optimistic decision maker believes that the nature intends to choose model parameters that favor the decision. 
As a consequence, there are four combinations of the ``max'' and ``min'' that can be set as the sense of optimization in \eqref{opt:FRO-F}. 
For the case where the objective function $f_k(x)$ is either convex or concave in the decision variable $x$, 
\eqref{opt:FRO-F} is a member of the following
list of problems that are divided into three categories:    
\begin{empheq}[left={\textrm{Category \rom{1} }\empheqlbrace}]{alignat=2}
& \underset{x\in X}{\textrm{min}} \textrm{ } \underset{k\in [K]}{\textrm{max}} \textrm{ } f_k(x) &&\qquad    f_k \textrm{ is concave}, \tag{MinMax-Concave} \label{opt:MinMax-Concave} \\
& \underset{x\in X}{\textrm{min}} \textrm{ } \underset{k\in [K]}{\textrm{min}} \textrm{ } f_k(x) &&\qquad     f_k \textrm{ is convex}, \tag{MinMin-Convex} \label{opt:MinMin-Convex} \\
&\underset{x\in X}{\textrm{max}} \textrm{ } \underset{k\in [K]}{\textrm{max}} \textrm{ } f_k(x)  &&\qquad    f_k \textrm{ is concave}, \tag{MaxMax-Concave} \label{opt:MaxMax-Concave}  \\
&\underset{x\in X}{\textrm{max}} \textrm{ } \underset{k\in [K]}{\textrm{min}} \textrm{ } f_k(x) &&\qquad     f_k \textrm{ is convex}, \tag{MaxMin-Convex} \label{opt:MaxMin-Convex}
\end{empheq}
\begin{empheq}[left={\textrm{Category \rom{2} }\empheqlbrace}]{alignat=2}
&\underset{x\in X}{\textrm{max}} \textrm{ } \underset{k\in [K]}{\textrm{min}} \textrm{ } f_k(x)  &&\qquad   f_k \textrm{ is concave}, \tag{MaxMin-Concave} \label{opt:MaxMin-Concave} \\
&\underset{x\in X}{\textrm{min}} \textrm{ } \underset{k\in [K]}{\textrm{max}} \textrm{ } f_k(x)  &&\qquad   f_k \textrm{ is convex}, \tag{MinMax-Convex} \label{opt:MinMax-Convex} 
\end{empheq}
\begin{empheq}[left={\textrm{Category \rom{3} }\empheqlbrace}]{alignat=2}
&\underset{x\in X}{\textrm{min}} \textrm{ } \underset{k\in [K]}{\textrm{min}} \textrm{ } f_k(x)  &&\qquad   f_k \textrm{ is concave}, \tag{MinMin-Concave} \label{opt:MinMin-Concave} \\
&\underset{x\in X}{\textrm{max}} \textrm{ } \underset{k\in [K]}{\textrm{max}} \textrm{ } f_k(x)  &&\qquad  f_k \textrm{ is convex}. \tag{MaxMax-Convex} \label{opt:MaxMax-Convex}
\end{empheq}
It can be shown that the four problems in Category \rom{1} are equivalent.
Similarly, problems in Category \rom{2} and problems in Category \rom{3} are equivalent.
Models in Category \rom{2} can be reformulated as convex optimization problems, and the models
in Category \rom{3} can be reformulated as convex maximization problems. However, for models in Category \rom{1}, the function $g(x)=\textrm{min}_{k\in [K]}\;f_k(x)$
is neither convex nor concave.  Therefore, problems in Category \rom{1} are the most challenging. In fact, even in the special case where $f_k(x)$ are piecewise-linear convex functions, problems in Category \rom{1} are NP-hard (see Theorem~1 in \cite{luo2019-PL-rob-opt}).

\subsection{Motivation}
The problem \eqref{opt:FRO-F} we studied in this paper is an important intermediate step 
towards developing algorithms for functionally-robust optimization (FRO) problems of the form:
\begin{equation}\label{opt:FRO}
\underset{x\in X}{\textrm{max}}\textrm{ } \underset{d\in\mathcal{D}}{\textrm{min}} \textrm{ } f(x,d),\tag{FRO}
\end{equation} 
where $d$ are parameter that specify function $f$,
and $\mathcal{D}$ is the ambiguity set of parameters $d$.
Algorithmic frameworks can be developed to solve \eqref{opt:FRO} by treating it as a semi-infinite program.  In an semi-infinite programming approach of solving \eqref{opt:FRO}, 
one solves a master problem and a separation problem in every iteration \cite{hettich1993_SIP-thy-methd-appl}.
The master problem is a relaxation of \eqref{opt:FRO} by considering a 
finite set $\widetilde{\mathcal{D}}$ of parameters, and the separation problem
\begin{equation*}
    \max_{d \in \mathcal{D}} f(\hat{x},d)
\end{equation*}
is solved at an incumbent solution $\hat{x}$ of the master problem to identify a new parameter $d^{\prime}\in\mathcal{D}$. A new cut based on $d^{\prime}$ is added to $\widetilde{\mathcal{D}}$ in the next iteration. 

The \eqref{opt:FRO-F} model also naturally arises in many practical applications. The following is an example from the situation where a diversified set of points are to be generated sequentially from a given set $X$. 
\begin{example}
\normalfont Let $X\subset\mathbb{R}^n$ be a bounded set, 
$\mathcal{D}\subset\mathbb{R}^n$ be a finite set of points and $f(x)=\|x-d\|_p$ be a norm function,
 e.g., $p=1$, $p=2$ or $p=\infty$ norm. Recall that a norm is a convex function. Consider the following problem:
\begin{equation}
\label{opt:maxmin-dist}
\underset{x\in X}{\textrm{max}}\;\underset{d\in \mathcal{D}}{\textrm{min}}\; \|x-d \|_p.
\end{equation}
The objective is to find the point $x$ from a bounded region $X$ that 
maximizes the minimum $p$-norm distance between $x$ and the points in $\mathcal{D}$.
This is a fundamental problem in developing a space filling design which has 
application in computer simulation experiments \cite{pronzato2012_space-filling-des,zou2004_sensor-deploy}. 
Input sample points generated from a space filling design
may lead to results with superior properties in statistical estimation \cite{pronzato2012_space-filling-des}.  
An approach to finding a space filling design is to generate the sample points sequentially.
In this approach, $\mathcal{D}$ is the current set of sample points generated.
Problem \eqref{opt:maxmin-dist} is solved to generate the next sample point.  
An engineering application of space filling design is in wireless communication,
where one wants to deploy sensors in a geographic region as uniformly as possible \cite{zou2004_sensor-deploy}.    
\end{example}

\subsection{Contributions}
We made the following contributions in this paper:
\begin{enumerate}
	\item We develop a geometric branch-and-bound algorithm (GB2) for solving a 
special case of \eqref{opt:FRO-F} where each candidate function is a piecewise-linear convex function.
The key idea of GB2 is the sequential partitioning of the feasible region
and imposes the branch-and-bound procedure. We prove finite convergence of this algorithm.
	\item Under the assumption that an oracle for solving a convex maximization problem over a polytope is available, we generalize the GB2 method for the case where the candidate functions are general convex functions. Specifically, we use an iterative linearization procedure in the GB2 method. We show that this algorithm convergences to an $\epsilon$-optimal solution in a finite number of steps.
	\item We provide computational results on the performance of the proposed algorithm.
\end{enumerate}

\subsection{Organization of this paper}
This paper is organized as follows: In Section~\ref{sec:liter-rev}, 
we provide a brief literature review on robust optimization, 
branch-and-bound methods for solving global optimization problems, 
and numerical methods for solving semi-infinite programming problems.
The development of the GB2 algorithm for a piece-wise linear case, and its convergence is given in Section~\ref{sec:geo-BB}. The generalization of the GB2 algorithm  for convex functions (G2B2) and the corresponding convergence results are given in Section~\ref{sec:GBB-smooth}. Computational experience with the help of numerical examples is discussed in Section~\ref{sec:num-study} for the G2B2 algorithm.

\subsection{Literature review}
\label{sec:liter-rev}
\subsubsection{Robust optimization}
Robust optimization (RO) models assume that the uncertainty is on the model parameters \cite{book_robust_opt2009}.
For instance, realizations of model parameters are drawn from an uncertainty set and the values are chosen to be adversarial to the decisions. Most robust convex optimization problems can be NP-hard. For example, robust SDPs with ellipsoidal uncertainty sets or polyhedral uncertainty sets are NP-hard
\cite{Ben-Tal-1998_rob-conv-opt,Ben-Tal2000_robustness}. However, robust linear programs with ellipsoidal uncertainty, polyhedral uncertainty and norm uncertainty, 
can be reformulated as second order cone programs (SOCPs) \cite{Ben-Tal-1999_rob-lin-prog}, 
linear programs (using linear duality theory) \cite{Ben-Tal-1999_rob-lin-prog}, 
and SOCPs, respectively \cite{bertsimas2005_opt-ineq-prob}. The optimization model \eqref{opt:FRO} studied in this paper has an essential difference from  the robust optimization models in the earlier papers, where the nominal problem is a convex-minimization problem. Problem \eqref{opt:FRO-F} is a robust counterpart of a convex-maximization problem.

\subsubsection{Relaxation and branch-and-bound methods in global optimization}
Branch-and-bound (B\&B) is an often used technique to solve global optimization problems. 
It is a scheme that successively refines a partition of the feasible set into subsets, 
and computes an upper and a lower bound of the objective value restricted on each subset. 
Based on comparing these bounds, subsets that do not contain a global optimal solution are removed from the search tree,
or a subset is chosen for further partitioning. This procedure is repeated until the optimality gap is as desired.   

Branch-and-bound methods have been widely applied to mixed-integer programs \cite{conforti-IP-2014},
and to global optimization of nonlinear programming models that have concave univariate, 
bilinear and linear fractional terms \cite{benson2006-frac-prog-conv-quad-form-func,benson2007-sum-ratio-frac-prog-concav-min,jiao2006-glob-opt-gen-lin-frac-prog-with-NL-constr,wang2005-BB-alg-glob-sum-lin-ratios,zamora1998-cont-glob-opt-struct-proc-sys-mod}.
Specifically, \cite{wang2005-BB-alg-glob-sum-lin-ratios} investigate the sum of linear ratios minimization problem with linear constraints. 
The feasible region is partitioned into sub-rectangles, and on each sub-rectangle a linear program is solved to evaluate a lower bound. 
This technique has been extend to solve general linear fractional programming problems with nonlinear constraints \cite{jiao2006-glob-opt-gen-lin-frac-prog-with-NL-constr}. Branch-and-bound algorithms that involve solving a sequence of convex programs are developed in \cite{benson2006-frac-prog-conv-quad-form-func,benson2007-sum-ratio-frac-prog-concav-min} for the fractional programming with convex
quadratic functions, as well as for the sum of linear ratios programming.
More studies about development of branch-and-bound methods for nonlinear programming are found in \cite{adjiman1998-alpha-BB-gen-twice-df-1,adjiman1998-alpha-BB-gen-twice-df-2,benson2007-BB-dual-bd-alg-sum-lin-ratio,shen2007-glob-opt-sum-gen-ploy-frac-func,sherali1998-glob-opt-nonconv-poly-prog-ratio,wang2005-glob-opt-gen-geom-prog}, etc.

An important application of branch-and-bound methods is in solving the mixed-integer nonlinear programming (MINLP) \cite{floudas1999,ryoo1995-glob-opt-NLP-MINLP-proc-design,sherali2001-glob-opt-nonconv-fact-prog,smith1996-glob-opt-gen-proc-model} models.
Tawarmalani and Sahinidis \cite{tawarmalani2004_glob-opt-MINLP-they-comp} develop a novel branch-and-bound  framework for the global optimization
of continuous and mixed-integer nonlinear programs. The relaxation procedure recursively decomposes nonlinear functions into 
fundamental mathematical operations and the algorithm considers relaxations based on these operations.
The algorithm uses duality theory for domain reduction, rectangular
partitioning for splitting the feasible set, and maximum weighted violation rule for branching variable selection. Some related research
on domain reduction and branching rules is in \cite{faria2011-novel-bound-contract-proc-bilin-MINLP-appl-water-manag,ryoo1996-branch-reduce-glob-opt,sahinidis1998-finite-alg-glob-min-sep-concav-prog,thakur1991-dom-contr-NLP,zamora1999-branch-contract-alg-concave-univ-bilin-lin-frac-term}, and \cite{bonami2013-branch-rule-conv-MINLP,leyffer2001-integ-SQP-BB-MINLP,linderoth2005-simp-BB-alg-quad-prog,ryoo1996-branch-reduce-glob-opt,sahinidis1998-finite-alg-glob-min-sep-concav-prog,tawarmalani-thesis2001}, respectively.

The geometric branch-and-bound (GB2) method developed in this paper for solving \eqref{opt:FRO-F}
has some similarity with the spatial branch-and-bound (SB2) method used in nonconvex optimization 
\cite{smith1999_spat-BB-nonconv-MINLP,pozo2010_spat-BB-opt-metabolic-netwk,
stein2013_enhan-spat-BB-glob-opt}. Both have the common idea of
sequentially partitioning the feasible region into sub-regions, and finding a lower and upper bound from the restricted problem (a convex optimization problem) in a sub-region
\cite{kirst2015_upper-bd-spat-BB,chen2017_spat-BC-nonconv-QCQP-compl-vars,
gerard2017_dive-for-spat-BB}. However, the G2B2 method developed in this paper differs from the spatial branch-and-bound approach in how it partitions the feasible set, and its evaluation of the lower and upper bounds. In the development of the G2B2 method we consider the structure of the robust optimization model.

\section{A Geometric Branch-and-bound Method for Piecewise-linear Candidate Functions}
\label{sec:geo-BB}
We first study a simplified problem of \eqref{opt:FRO-F}, where 
each candidate function $f_k(x)$ is a piecewise-linear function in $x$.
Specifically, the candidate function is represented as:
 \begin{equation}
 f_k(x):=f(x,d_k)=\textrm{max}\{a^{ki\top}x + b^{ki} \;|\; i\in[I_k] \}, \qquad \forall k\in[K],
 \end{equation}
where $I_k$ is the number of linear pieces of $f_k$ and $a^{ki}$, $b^{ki}$ are the coefficients
of the $i^{\textrm{th}}$ piece of $f_k$. In this simplified case, the \eqref{opt:FRO} problem
becomes the following problem:
\begin{equation}\label{opt:PL}
\underset{x\in X}{\textrm{max}}\;\underset{k\in[K]}{\textrm{min}}\;\textrm{max}\{a^{ki\top}x + b^{ki} \;|\; i\in[I_k] \}.       \tag{PL}
\end{equation}
Let $f(x)=\textrm{min}_{k\in[K]}\; f_k(x)$ be the objective function in the outer optimization problem of \eqref{opt:PL}.
We develop a geometric branch-and-bound (GB2) algorithm for solving \eqref{opt:PL},
which takes advantage of the problem structure and geometry.

\subsection{An outline of the GB2 algorithm}
\label{sec:GB2-outline}
In the GB2 algorithm, the feasible set of \eqref{opt:PL} 
is recursively divided into subsets that are specified by certain linear pieces of candidate functions,
and the objective is optimized in the selected sub-regions to obtain lower and upper bounds of \eqref{opt:PL}. 
In the GB2 algorithm, a node represents a polytope in $X$. This polytope is characterized by
a subset of candidate functions and their active linear pieces, i.e., the linear piece that can give the function
value in the polytope. For example, suppose a node $P$ is specified by candidate functions
$\{f_{k_1},\ldots,f_{k_r}\}$, where $\{k_1,\ldots,k_r\}\subset [K]$, and the active linear pieces (that define $P$) 
are $i_1,\ldots,i_r$ corresponding to $f_{k_1},\ldots,f_{k_r}$, respectively. 
The polytope represented by $P$ is given by $P=\cap^{r}_{s=1}Q^{k_si_s}$, where the polytope $Q^{k_si_s}$
is defined as:
\begin{equation}\label{def:Q}
Q^{k_si_s}=\Set*{x\in X}{a^{k_si_s}x+b^{k_si_s}\ge  a^{k_sj}x+b^{k_sj} \quad \forall j\in[I_{k_s}]\setminus\{i_s\}}.
\end{equation} 
The polytopes $\{Q^{ki}:\;i\in[I_k]\}$ form a partition of $X$ for every $k\in[K]$, 
and we have $f_k(x)=a^{ki}x+b^{ki}\;\;\forall x\in Q^{ki}$.     

To better describe the GB2 algorithm, the notation $P$ is used 
to label a node and the polytope represented by the node. 
As in the standard branch-and-bound algorithm, the GB2 algorithm manages a branch-and-bound tree. 
The label $\mJ[P]$ denotes the set of branching information that is needed to define the node $P$.
It can be explicitly written as $\mJ[P]=\{ (k_s,i_s)\}^l_{s=1}$, where $k_s\in[K],\; i_s\in[I_{k_s}]$ for $s\in[l]$. 
Let $\mK[P]=\{k_s\}^l_{s=1}$ be the set of function indices associated with $P$.
The branching information in $\mJ[P]$ indicates that in the $s^{\textrm{th}}$ branching iteration associated with $P$,
a candidate function $f_{k_s}$ is selected, 
and the ${i_s}^{\textrm{th}}$ linear piece of $f_{k_s}$ is selected as the dominate linear piece.
The scheme of the GB2 algorithm is summarized as follows:
\begin{enumerate}
	\item Find a leaf node $P$ of the GB2 tree. 
	\item Include a candidate function $f_r$ to the scope of $P$ if it is not yet considered in $P$.
	\item Branch $P$ into $I_r$ child nodes in a way that the $i^{\textrm{th}}$ 
		  linear piece of $f_r$ is active in the $i^{\textrm{th}}$ child node.
	\item Compute upper and lower bounds for each child node.
	\item Update the global upper and lower bounds. Pune leaf nodes that have worse upper bounds than the global lower bound.
	\item Repeat the above steps.  
\end{enumerate}  
In Section~\ref{sec:upper-lower-bd-compt} and Section~\ref{sec:GB2-alg}, 
we give detailed explanation of each step in the above scheme.

\subsection{Upper and lower bounds computation}
\label{sec:upper-lower-bd-compt}
In Step 4 of the GB2 scheme, we need to compute a lower and an upper bounds of a node $P$.
The two bounds are computed from the following linear program:
\begin{equation}\label{opt:node_LP}
	\begin{array}{cl}
	 \textrm{max} & \theta \\
	\textrm{s.t.} & \theta \le a^{ki} x + b^{ki}, \qquad  \forall (k,i)\in\mJ \\
			      & x\in P. 
	\end{array}
\end{equation}
The upper bound $U[P]$ is the optimal value of \eqref{opt:node_LP}, 
and the lower bound $L[P]$ is the function value $f(x^*)$, 
where $x^*$ is the optimal solution of \eqref{opt:node_LP}.
In Lemma~\ref{lemma:branch-bound-lemma} we prove the validity of these bounds.
The GB2 algorithm also keeps track of a global lower bound $L$ and a global 
upper bound $U$ of \eqref{opt:PL}. 
The global bounds are updated at the end of each iteration in the GB2 algorithm. 
\begin{lemma}
\label{lemma:branch-bound-lemma}
At the end of iteration $m$ ($m\ge 0$) of Algorithm~\ref{alg:geom-BB-PL}, the following properties hold:
\begin{enumerate}
	\item[\emph{(a)}] For any node $P$ in the current branch-and-bound tree, $U[P]\ge \emph{max}_{x\in P}\; f(x)$,
					where $U[P]$ is computed by solving \eqref{opt:node_LP}.
	\item[\emph{(b)}] The optimal value of \eqref{opt:PL} is given by
					$\emph{val(\ref{opt:PL})}=\emph{max}_{P\in T^m}\emph{max}_{x\in P}f(x)$, 
					where $T^m$ is the set of leaves in the GB2 tree at the end of the iteration $m$.
	\item[\emph{(c)}] The values of $U$ and $L$ computed in Step 4 of the GB2 algorithm (Algorithm~\ref{alg:geom-BB-PL}) 
					are an upper and a lower bound on the optimal value of \eqref{opt:PL}.   
\end{enumerate}
\end{lemma}
\begin{proof}
(a): For any node $P$ in the current branch-and-bound tree, we have
\begin{equation}
\begin{aligned}
U[P]&=\textrm{max}\;\{\theta:\;  \theta\le a^{ki}x + b^{ki}\quad \forall (k,i)\in\mJ[P], \;x \in  P \} \\
&=\textrm{max}\;\{\theta:\;  \theta\le f_k(x) \quad \forall k\in \mK[P], \;x \in P \} \\
&=\textrm{max}\; \{ \textrm{min}\;\{ f_k(x):\; k\in \mK[P] \}\quad x \in P \} \\
&\ge \textrm{max}\;\{f(x): \;x \in  P \},
\end{aligned}
\end{equation}
where the last inequality uses the property that 
$\textrm{min}\;\{ f_k(x):\; k\in \mK[P] \}\ge \textrm{min}\;\{ f_k(x):\; k\in K \}$
for any $x\in X$, since $\mK[P]$ is a subset of $K$.

(b): Prove by induction on iteration $m$. Clearly, the equality holds for $m=0$. 
Suppose it holds at the end of the $m^{\textrm{th}}$ iteration. 
If the algorithm does not terminate, then it picks a leaf $P^{\prime}\in T^m$ such that $U[P^{\prime}]=U$, 
and branch $P^{\prime}$ into $P^{\prime}_1,\ldots,P^{\prime}_I$. 
By the induction hypothesis, we have
\begin{equation}\label{eqn:val(PL)=maxmax}
\begin{aligned}
\textrm{val\eqref{opt:PL}} &= \textrm{max}_{P\in T^m} \textrm{max}_{x\in P}\;f(x) 
			  = \textrm{max}_{P\in \widetilde{T}^{m+1}}\textrm{max}_{x\in P}\; f(x),
\end{aligned}
\end{equation}
where $\widetilde{T}^{m+1}=(T^m\setminus \{P^{\prime}\})\cup \{P^{\prime}_i\}^I_{i=1}$.
The last equality in \eqref{eqn:val(PL)=maxmax} is due to that $P^{\prime}=\cup^I_{i=1}P^{\prime}_i$ 
Since nodes in the set $\widetilde{T}^{m+1}\setminus T^{m+1}$ are pruned at the end of iteration $m+1$.
It follows that for every $P\in \widetilde{T}^{m+1}\setminus T^{m+1}$, 
there exists a $\widetilde{P}\in T^{m+1}$ such that $L[\widetilde{P}] \ge U[P]$. 
Let $x_{\widetilde{P}}$ be the optimal solution of the linear program \eqref{opt:PL} at node $\widetilde{P}$.
Since $f(x_{\widetilde{P}})=L[\widetilde{P}]$ and $U[P] \ge \textrm{max}_{x\in P} f(x)$,
we have $f(x_{\widetilde{P}})\ge\textrm{max}_{x\in P}f(x)$.
Therefore, the following holds:
\begin{equation}\label{eqn:maxmax=maxmax}
\textrm{max}_{P\in \widetilde{T}^{m+1}}\textrm{max}_{x\in P}\; f(x) = \textrm{max}_{P\in T^{m+1}}\textrm{max}_{x\in P}\; f(x).		     
\end{equation}
Equations \eqref{eqn:val(PL)=maxmax} and \eqref{eqn:maxmax=maxmax} imply that (b) holds for iteration $m+1$.

(c): From the algorithm, we have
\begin{equation}
\begin{aligned}
&L=\textrm{max}_{P\in T^m} L[P] = \textrm{max}_{P\in T^m} f(x_{P}) \le \textrm{max}_{x\in X} f(x), \\
&U=\textrm{max}_{P\in T^m}\;U[P] \ge \textrm{max}_{P\in T^m}\textrm{max}_{x\in P}\; f(x), 
\end{aligned}
\end{equation}
where $x_{P}$ is the optimal solution of the linear program \eqref{opt:node_LP} at the node $P$.
\end{proof}

\subsection{The GB2 algorithm}
\label{sec:GB2-alg}
An iteration in the algorithm consists of the following five major steps: 
leaf selection, function selection, branching, bound updating, and pruning. 
In the leaf-selection step, we pick a leaf of the current GB2 tree, 
such that the corresponding upper bound value is the maximum among all leaves.   
Suppose the selected leaf is $P$ at Level $l$.
In the function-selection step, we select a function $f_{k_{l+1}}$, 
where $k_{l+1}\in[K]\setminus K_{\J}$. The function $f_{k_{l+1}}$
satisfies that $\textrm{max}_{x\in X}f_{k_{l+1}}\le
\textrm{max}_{x\in X}f_{k}$ for any $k\in[K]\setminus \mK[P]$.
We note that the choice of $k_{l+1}$ is empirical. The algorithm remains valid,
if the index $k_{l+1}$ is selected arbitrarily from the set $[K]\setminus\mK[P]$. 

In the branching step, the polytope $P$ is partitioned into $I_{k_{l+1}}$ 
sub-polytopes (child nodes). The $i^{\textrm{th}}$ sub-polytope is given by 
$P\cap Q^{k_{l+1}i}$ for $i\in[I_{k_{l+1}}]$.
Denote the $i^{\textrm{th}}$ child node as $P_i$, and set $\mJ[P_i]=J\cup\{(k_{l+1},i)\}$.
In the bound-updating step, we compute the upper bound and the lower bound 
associated with the child node based on the optimal solution of the linear program \eqref{opt:node_LP},
with the feasible set of \eqref{opt:node_LP} being the polytope corresponding to the child node.
The algorithm then updates global lower and upper bounds $L$ and $U$ as
$L = \textrm{max}\{ L[P]: P\textrm{ is a leaf of the current GB2 tree} \}$  
and $U=\textrm{max}\{ U[P]: P\textrm{ is a leaf of the current GB2 tree} \}$, respectively. 
In the pruning step, we remove all current leaves of which the upper bound value 
is no greater than the current global lower bound $L$. 
Finally, we repeat above steps until the optimality gap $U-L$ is less than the tolerance value
or no leaf node remains.  
The GB2 algorithm is given in Algorithm~\ref{alg:GBB}. 
We show in Theorem~\ref{thm:GBB-finite-converg} that the GB2 algorithm can identify the global
optimal solution of the problem \eqref{opt:PL} in finitely many iterations. 

\begin{theorem}\label{thm:GBB-finite-converg}
The geometric branch-and-bound algorithm (Algorithm~\ref{alg:geom-BB-PL}) terminates in finitely many (outer) iterations, 
and returns a global minimizer to \eqref{opt:PL}. 
\end{theorem}
\begin{proof}[Proof of Theorem~\ref{thm:GBB-finite-converg}]
By Lemma~\ref{lemma:branch-bound-lemma}, $U$ and $L$ are upper and lower bounds 
of $\textrm{val}(\ref{opt:PL})$ at any iteration. 
If at some iteration $L=U$, an optimal solution is found and the optimal value of \eqref{opt:PL} is $L$. 
Notice that, if $P$ is a level $K$ leaf node, then $L[P]=U[P]$.
If all leaf nodes are at level $K$, then is is guaranteed that $U=L$. 
It takes finitely many iterations to reach the status that the GB2 tree has only $K$ level leaves.
\end{proof}

\begin{algorithm}
	\caption{A geometric branch-and-bound algorithm to solve (\ref{opt:PL}). }\label{alg:geom-BB-PL}
	\label{alg:GBB}
	\begin{algorithmic}
	\STATE {\bf Input}: An instance of the problem \eqref{opt:PL}.     
	\STATE {\bf Output}: A global solution to \eqref{opt:PL}.
	
	\STATE{(Initialization) Define the root node $P^0$ as: $P^0\gets X$, $\mJ[P^0]\gets\emptyset$.
				       Set $L\gets -\infty$ and $U\gets\infty$.}
	\WHILE{$L<U$}	
			\STATE 1. Branching node selection: Pick a leaf node $P$ of the current GB2 tree, 
			\STATE \hspace{4mm} such that $U[P]=\textrm{max}\{ U[P^{\prime}]:  P^{\prime}\textrm{ is a leave of the current tree} \}$
			\STATE 2. Function selection: Let $k_{l+1}$ be a randomly chosen function index from $[K]\setminus\mK[P]$.
			\STATE 3. Branching: Create $I_{k_{l+1}}$ child nodes of $P$, denoted by $\Set*{ P_i}{i\in[I_{k_{l+1}}]}$.  
			\STATE \hspace{4mm} Set $\mJ[P_i]=\mJ[P]\cup\{(k_{l+1},i)\}$, and $P_i=P\cap Q^{k_{p+1}i}$.	
			\STATE 4. Upper and lower bound evaluation: For each new polytope $P_i$, solve the problem:
					\begin{displaymath}
					\begin{array}{cl}
					\textrm{max}& \theta \\
					\textrm{s.t.} & \theta \le a^{kj} x + b^{kj} \qquad  (k,j)\in\mJ[P_i], \\
			 		& x\in P_i. 
					\end{array}
					\end{displaymath}
 			\STATE  \hspace{4mm} Let $(\theta^i,x^i)$ be the optimal solution of the above LP. 
			\STATE  \hspace{4mm} Set $U[P_i]\gets\theta^i$, $L[P_i]\gets f(x^i)$
								    and $x_{P_i}\gets x^i$ for each $i\in[I_{k_{l+1}}]$.  
			\STATE  \hspace{4mm} Let $L\gets \textrm{max}\{L[P^{\prime}]: P^{\prime}\textrm{ is a leaf}  \} $ 
									and $U\gets \textrm{max}\{ U[P^{\prime}]: P^{\prime}\textrm{ is a leaf} \}$.
			\STATE 5. Pruning and update of global bounds: 	
				     \bindent				
			            \FOR{$P\in\{\textrm{current leaf nodes}$\}}					
					 	\IF{$U[P] \le L$}							
							\STATE  Prune the node $P$.							
						\ENDIF
						\IF{$L[P]=U$}							
							\STATE  Set $L\gets L[P]$, $x^*\gets x_{P}$, \bf{Stop}. 							
						\ENDIF
					 \ENDFOR		
					 \eindent								
	\ENDWHILE
	\STATE {\bf Return} $x^*$		
	\end{algorithmic}	 	
\end{algorithm}

We create a numerical instance of \eqref{opt:PL} (Numerical Instance 1) to better illustrate the GB2 algorithm and the GB2 tree structure.
\begin{instance}\label{exp:PL-nonconvex}
\normalfont Let $X=\{(x_1,x_2)\in\mathbb{R}^2:\;0\le x_1\le 10,\; 0\le x_2\le 10\}$, $K=2$ and $I=3$.
Let $x=(x_1,x_2)$ be a two dimensional vector. Let
\begin{equation}
\begin{aligned}
& f_1(x) = \textrm{min}\{ 16.36 x_1 - 7.73x_2 + 243.18,\; -13.75 x_1 - 13.75x_2 + 393.75, \\
&\qquad\qquad\qquad						 -0.5x_1 + 26x_2 + 142  \}; \\
& f_2(x)=\textrm{min}\{ 3.91x_1 + 1.2x_2 + 215.43,\; 1.47x_1 - 7.35x_2 + 283.82,\\
&\qquad\qquad\qquad 5.6x_1+0.8x_2+210 \}. \nonumber
\end{aligned}
\end{equation}
The \eqref{opt:PL} problem becomes $\textrm{max}_{x\in X}\textrm{min}\{f_1(x),f_2(x)\}$.
Consider the GB2 tree associated with this problem.
The root node is $P^0=X$. 
In the algorithm, we first include the candidate function $f_1$, 
which creates three child nodes (Level 1), one for each of the linear piece
of $f_1$. These are denoted by: $P^{1}_1$,
$P^{1}_2$ and $P^{1}_3$, where the superscript is the level index 
and the subscript is the index for the dominating linear piece. 
The information labels associated with the nodes are written as $\mJ[P^1_1]=\{(1,1)\}$, 
$\mJ[P^1_2]=\{(1,2)\}$ and $\mJ[P^1_3]=\{(1,3)\}$, respectively. The two entries in each label indicate
the index of the candidate function and the index of the active linear piece, respectively.
Based on this information, the polytopes associated with the three nodes are:
\begin{displaymath}
\begin{aligned}
&P^1_1=\Set*{x\in X}{
\begin{array}{l}
16.36 x_1 - 7.73x_2 + 243.18\ge -13.75 x_1 - 13.75x_2 + 393.75 \\
16.36 x_1 - 7.73x_2 + 243.18\ge -0.5x_1 + 26x_2 + 142
\end{array}
},  \\
&P^1_2=\Set*{x\in X}{
\begin{array}{l}
-13.75 x_1 - 13.75x_2 + 393.75 \ge 16.36 x_1 - 7.73x_2 + 243.18\\
-13.75 x_1 - 13.75x_2 + 393.75 \ge -0.5x_1 + 26x_2 + 142
\end{array}
},  \\
&\textrm{ and }P^1_3=\Set*{x\in X}{
\begin{array}{l}
 -0.5x_1 + 26x_2 + 142 \ge 16.36 x_1 - 7.73x_2 + 243.18 \\
 -0.5x_1 + 26x_2 + 142 \ge -13.75 x_1 - 13.75x_2 + 393.75
\end{array}
}. 
\end{aligned}
\end{displaymath}
The three nodes partition $X$ into three sub-polytope labeled as P1, P2, P3 in Figure~\ref{fig:partition_X}(a).
The lower and upper bounds restricted in the three sub-polytopes are: $(L[P^1_1],U[P^1_1])=(298.52,406.78)$,
$(L[P^1_2],U[P^1_2])=(283.82,393.75)$, \newline $(L[P^1_3],U[P^1_3])=(227.43,402.0)$, respectively. 
The global lower and upper bounds are given by $(L,U)=(298.52,406.78)$.
Since the global upper bound is attained at node $J^1_1$, the algorithm picks this node to fathom 
in the next iteration. The candidate function $f_2$ is included in the scope to branch $J^1_1$ into
three child nodes denoted as: $J^2_1$, $J^2_2$ and $J^2_3$,
where the node information are $\mJ[P^2_1]=\{(1,1),(2,1)\}$, $\mJ[P^2_2]=\{(1,1),(2,2)\}$
and $\mJ[P^2_3]=\{(1,1),(2,3)\}$, respectively. The polytope associated with $P^2_1$ is:
\begin{displaymath}
P^2_1=\Set*{x\in P^1_1}{
\begin{array}{l}
3.91x_1 + 1.2x_2 + 215.43 \ge 1.47x_1 - 7.35x_2 + 283.82 \\
3.91x_1 + 1.2x_2 + 215.43 \ge 5.6x_1+0.8x_2+210
\end{array}
}.
\end{displaymath}
The polytopes associated with  $P^2_2$ and  $P^2_3$ can be constructed similarly
based on their node information. We can verify that the sub-polytope associated to $P^2_1$ is empty.
The sub-polytopes associated to $P^2_2$ and $P^2_3$ are labeled as P12
and P13 in Figure~\ref{fig:partition_X}(b), respectively. The lower and upper bounds restricted in the two sub-polytopes are
$(L[P^2_2], U[P^2_2])=(298.52,298.52)$ and 
$(L[P^2_3],U[P^2_3])=(272.40,272.40)$, respectively.
The global bounds are updated as 
\begin{displaymath}
\begin{aligned}
&L=\textrm{max}\{L[P^1_2],\;L[P^1_3],\; L[P^2_2],\;L[P^2_3]\}=298.52, \\
&U=\textrm{max}\{U[P^1_2],\;U[P^1_3],\; U[P^2_2],\;U[P^2_3]\}=298.52.
\end{aligned}
\end{displaymath}
Since $L=U$, the algorithm terminates. 
\end{instance}
\begin{figure}
\centering
\subfigure[]{\includegraphics[width=0.45\textwidth]{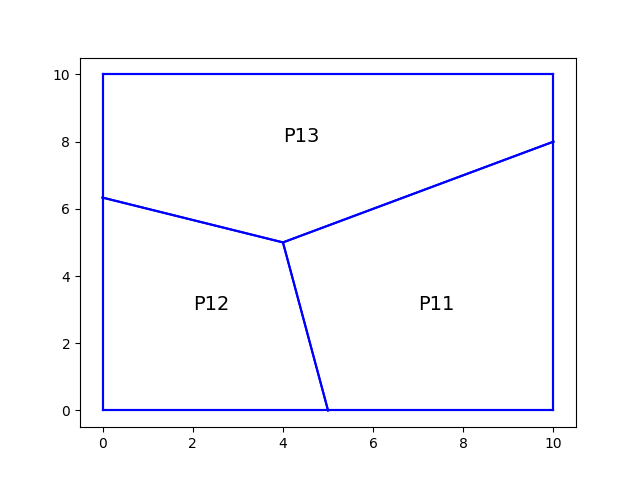}}
\subfigure[]{\includegraphics[width=0.45\textwidth]{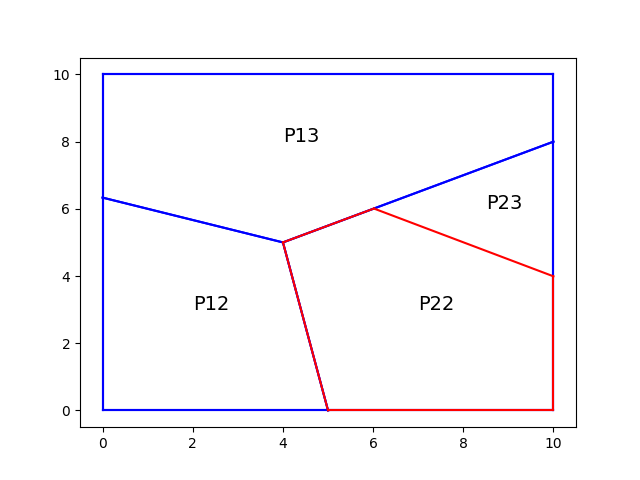}}	
\caption{Partition of $X$ into sub-polytopes in the GB2 algorithm, when Algorithm~\ref{alg:geom-BB-PL} 
			is applied to Numerical Instance~\ref{exp:PL-nonconvex}. }
\label{fig:partition_X}
\end{figure}

\begin{figure}
	\centering
	\includegraphics[width=0.6\textwidth, trim=0.1cm 10cm 0.1cm 4.5cm]{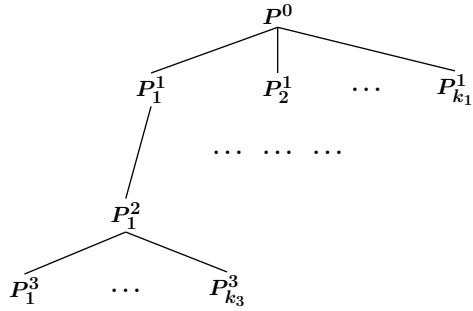}
	\caption{The structure of the GB2 tree. A previous leaf $P^2_1$ at the level $2$
			branches into $I_{k_3}$ child nodes which are leaves of the current GB2 tree.
			The dot line under the node $P^1[J_1]$ indicates that there are some 
			branches and nodes not shown in the figure.}
	\label{fig:BBtree}
\end{figure}

\section{A Generalized Geometric Branch-and-bound Method for General Convex Candidate Functions}
\label{sec:GBB-smooth}
We generalize the GB2 method given in Section~\ref{sec:geo-BB} to solve \eqref{opt:FRO-F}
for the case that each $f_k$ is general convex function.  
The main idea of the generalized geometric branch-and-bound algorithm (G2B2)
developed in this section is to work with a piecewise-linear approximation of each candidate function $f_k(x)$
and keep refining the approximation within the geometric branch-and-bound framework. 
The refinement step assumes existence of an oracle that maximizes a convex function over a polytope to a desired accuracy.
For simplicity, the algorithm assumes that each candidate function has an initial linear approximation.
At each iteration, a leaf node of the current G2B2 tree is selected for branching. 
Then a new piece of linear function is added to refine the envelope of $f_k$ for some $k\in[K]$. 
The node will branch with respect to the active region of the new linear piece.
The generalized GB2 algorithm is briefly summarized as follows:
\begin{enumerate}
	\item Find a leaf node $P$ of the G2B2 tree.
	\item If the gap between the lower and upper bounds at $P$ does not meet the tolerance, 
		  and there exists a candidate function $f_r$ that is not well approximated, add a new linear piece $p$ to refine the approximation of $f_r$.
	\item Branch $P$ into two child nodes $P^{\prime}$, $P^{\prime\prime}$. 
		  The node $P^{\prime}$ is defined such that $p$ is active, 
		  i.e., the value given by the linear piece $p$ is greater than or equal to the value given by
		  other linear pieces of the function $f_r$ for all $x$ in the region of $P^{\prime}$. 
		  The node $P^{\prime\prime}$ 
	         is defined such that the previous active linear piece of $f_r$ at $P$ remains active at $P^{\prime\prime}$. 
	\item Compute upper and lower bounds at $P^{\prime}$ and $P^{\prime\prime}$.
	\item Update the global upper and lower bounds. Pune leaf nodes that have worse upper bounds than the global lower bound.
	\item Repeat the above steps.
\end{enumerate}

\subsection{Upper and lower bounds computation}
\label{sec:G2B2-bounds-comp}
We give a method to evaluate a lower and an upper bounds for each leaf node in the G2B2 tree. 
The bounds evaluation depends on the approximation error of a candidate function $f_k$ with
a piecewise-linear function. The approximation error is defined in Definition~\eqref{def:e-acc}.
To refine the piecewise-linear approximation for each candidate function,
we need an oracle (Assumption~\ref{ass:conv-max-oracle}) 
to solve a separation problem which is a maximization problem of a convex function.
\begin{definition}[$\epsilon$-accuracy]\label{def:e-acc}
A candidate function $f_k$ is $\epsilon$-accurate at a node $P$ of the current
G2B2 tree if $f_k(x)-\widehat{f}_k(x)\le\epsilon,\;\forall x\in P$, where $\widehat{f}_k$ is a piecewise-linear envelope  
approximation of $f_k$ at $P$.   
\end{definition}
\begin{assumption}\label{ass:conv-max-oracle}
\normalfont For the function $f_k(x)$, there exists an oracle to find an $\epsilon$-optimal solution to the problem:
$\textrm{max}_{x\in P}f_k(x)-a^{\top}x-b$, where $P$ is a polytope. 
\end{assumption}
\noindent 
For any node $P$ of the G2B2 tree and $k\in[K]$, 
let $f^{P}_k$ be the linear function used to approximate $f_k$ at node $P$, 
and write $f^{P}_k(x)=a^{P}_k x + b^{P}_k$. 
Let $\mK^{\epsilon}_{P}$ be a subset of $[K]$, such that for any 
$k\in \mK^{\epsilon}_{P}$, the function $f_k$ is $\epsilon$-accurate in $P$. 
In the algorithm the set $\mK^{\epsilon}_{P}$ is inherited first from the mother node of $P$
and it will be updated when processing node $P$. For any node $P$, there is a set $S_{P}$ 
associate with $P$:  $\mS_{P}=\big\{(x^{P,\epsilon}_k, v^{P,\epsilon}_k): 
k\in [K]\setminus \mK^{\epsilon}_{P} \big\}$,
where $x^{P,\epsilon}_k$ and $v^{P,\epsilon}_k$ are the $\epsilon$-optimal solution
and $\epsilon$-optimal value of the problem:
\begin{equation}\label{opt:sep-D}
\underset{x\in P}{\textrm{max}}\; f_k(x)-a^{P}_k x - b^{P}_k  \tag{SEP-$P$}
\end{equation} 
for any $k\in [K]\setminus \mK^{\epsilon}_{P}$. Note that \eqref{opt:sep-D} 
is a convex maximization problem over a polytope, 
which is solved using an oracle (Assumption~\ref{ass:conv-max-oracle}). 
A node $P$ is further associated with values $L[P]$, $U[P]$ and $(x^{P},\theta^{P})$.  
The $x^P$, $\theta^P$ are the optimal solution 
and the optimal value of the following linear program:
\begin{equation}\label{opt:LP-D}
	\begin{array}{rll}
		\textrm{max} & \theta & \\
		 \textrm{s.t.}  & \theta\le a^{P}_k x + b^{P}_k + v^{P,\epsilon}_k + \epsilon & k\in [K]\setminus \mK^{\epsilon}_P \\
		 			  & \theta\le a^{P}_k x + b^{P}_k + \epsilon & k\in \mK^{\epsilon}_P  \\
		 			  & x\in P.  
	\end{array}
	\tag{LP-$P$}
\end{equation}
The $L[P]=f(x^P)$ and $U[P]=\theta^P$ are a lower bound and an upper bound of the optimization problem
$\textrm{max}_{x\in P}f(x)$, respectively (see Proposition~\ref{prop:theta-upperbound-maxmin-fk}).
\begin{proposition}\label{prop:theta-upperbound-maxmin-fk}
For any node $P$ in the G2B2 tree, we have 
\begin{equation}
\theta^{P}\ge \emph{max}_{x\in P}\emph{min}_{k\in [K]}\; f_k(x).
\end{equation} 
The following value
\begin{equation}
U := \emph{max}\{\theta^{P}\; |\; P\in\emph{leaf nodes of the G2B2 tree}\}.
\end{equation}
is a global upper bound of \eqref{opt:FRO-F}.
\end{proposition}
\begin{proof}
Note that the candidate function $f_k$ is approximated by the linear function $a^{P}_kx + b^{P}_k$ in region $P$.
We first show that the linear functions $a^{P}_k x + b^{P}_k + v^{P,\epsilon}_k + \epsilon$ 
and $a^{P}_k x + b^{P}_k + \epsilon$ majorize the function $f_k(x)$ for any $k\in [K]\setminus \mK^{\epsilon}_{P}$ and 
any $k\in \mK^{\epsilon}_{P}$, respectively. For any $k\in [K]\setminus \mK^{\epsilon}_{P}$,
since $v^{P,\epsilon}_k$ is an $\epsilon$-optimal value of \eqref{opt:LP-D}, we have
\begin{displaymath}
\begin{aligned}
&a^{P}_k x + b^{P}_k + v^{P,\epsilon}_k + \epsilon\ge 
a^{P}_k x + b^{P}_k + \textrm{max}_{x^{\prime}\in P}\;f_k(x^{\prime})-a^{P}_kx^{\prime}-b^{P}_k \\
&\ge a^{P}_k x + b^{P}_k+f_k(x)-a^{P}_kx-b^{P}_k\ge f_k(x) \qquad \forall x\in P. 
\end{aligned}
\end{displaymath}
Similarly for any $k\in \mK^{\epsilon}_{P}$, we have $a^{P}_k x + b^{P}_k+\epsilon\ge f_k(x)\;\;\forall x\in P$.
Let $x^*$ be the optimal solution of $\textrm{max}_{x\in P}\textrm{min}_{k\in [K]}\;f_k(x)$. We have:
\begin{equation}
	\begin{aligned}
		&\textrm{max}_{x\in P}\textrm{min}_{k\in [K]}\;f_k(x) = \textrm{min}_{k\in [K]}\;f_k(x^*)  \\
		&\le\textrm{min}\left\{ \textrm{min}_{k\in [K]\setminus \mK^{\epsilon}_{P}}\;f_k(x^*),\; \textrm{min}_{k\in \mK^{\epsilon}_{P}}\;f_k(x^*)  \right\} \\
		&\le \textrm{min}\left\{ \textrm{min}_{k\in [K]\setminus \mK^{\epsilon}_P}\;a^{P}_k x^* + b^{P}_k + v^{P,\epsilon}_k + \epsilon,\; \textrm{min}_{k\in \mK^{\epsilon}_P}\;a^{P}_k x^* + b^{P}_k +  \epsilon  \right\}    
		\le \theta^{P},
	\end{aligned}
\end{equation}
where in the last inequality, we use the property that $\theta^{P}$ is the optimal value of the 
linear program \eqref{opt:LP-D}. 
Since the leaf nodes of the G2B2 tree form a partition of the feasible set $X$, it is clear to see that $U$ is 
a global upper bound of \eqref{opt:FRO-F}.   
\end{proof}

\subsection{The G2B2 algorithm}
\label{sec:G2B2-alg}
The G2B2 algorithm selects the leaf node having the maximum upper bound over all leaf nodes.
Suppose that node $P$ is selected for branching.
If $[K]\setminus \mK^{\epsilon}_P$ is non-empty, 
it indicates that some candidate functions are not approximated with $\epsilon$-accuracy within the node.
In this case, we add a new linear piece to a candidate function from the set $[K]\setminus \mK^{\epsilon}_P$. 
Specifically, we choose a candidate function index 
$k^*\in [K]\setminus \mK^{\epsilon}_P$ such that 
$v^{P,\epsilon}_{k^*}=\textrm{max}\big\{v^{P,\epsilon}_k: k\in [K]\setminus \mK^{\epsilon}_P \big\}$
(Recall that $v^{P,\epsilon}_k$ is an $\epsilon$-optimal value of \eqref{opt:sep-D}).
Then using a sub-gradient $g$ of the function $f_{k^*}$ at the point $x^{P,\epsilon}_{k^*}$,
we add a linear function: $\psi(x)=g^{\top}(x-x^{P,\epsilon}_{k^*} ) + f_k(x^{P,\epsilon}_{k^*})$ 
to refine the approximation of $f_k$.     
The node $P$ is then branched into two children $P_1$ and $P_2$, where
\begin{equation}
\label{opt:master-delta}
	\begin{aligned}
		P_1&=P\cap\big\{ a^{P}_{k^*}x + b^{P}_{k^*}\ge g^{\top} \big(x-x^{P,\epsilon}_{k^*} \big) + f_k\big(x^{P,\epsilon}_{k^*}\big)  \big\},\\
		P_2&=P\cap\big\{ a^{P}_{k^*}x + b^{P}_{k^*}\le g^{\top}\big(x-x^{P,\epsilon}_{k^*} \big) + f_k\big(x^{P,\epsilon}_{k^*}\big)  \big\}.
	\end{aligned}
\end{equation}
and we make the following initialization for $P_1$ and $P_2$, respectively:
\begin{equation}\label{eqn:def_K_f}
\begin{aligned}
&  f^{P_1}_k\gets f^{P}_k\quad \forall k\in[K], \\
& f^{P_2}_{k^*}\gets g^{\top}\big(x-x^{P,\epsilon/2}_{k^*} \big) + f_k\big(x^{P,\epsilon/2}_{k^*}\big),
\qquad   f^{P_2}_k\gets f^{P}_k \quad \forall k\in [K]\setminus\{k^*\}.
\end{aligned}
\end{equation} 
The purpose of the above assignment in the G2B2 algorithm is to define the approximated function of each candidate function
for $P_1$ and $P_2$.
For each new leaf node $P_i$ ($i=1,2$), solve the following $|\mK^{\epsilon}_{P}|$ separation problems:
\begin{equation}\label{opt:max_fk-ax-b}
\textrm{max}_{x\in P_i}\; f_k(x)-a^{P_i}_k x - b^{P_i}_k\qquad k\in \mK^{\epsilon}_P.
\end{equation} 
Let $x^{P_i,\epsilon}_k$ and $v^{P_i,\epsilon}_k$ be the 
$\epsilon$-optimal solution, and the $\epsilon$-optimal value of the above problem. 
Let $\mK^{\epsilon}_{P_i}=\big\{ k\in \mK^{\epsilon}_{P}\;|\; v^{P_i, \epsilon}_k>\epsilon \big\}$
and $\mS_{P_i}=\big\{ (x^{P_i,\epsilon}_k, v^{P_i,\epsilon}_k)\;|\;k\in [K]\setminus \mK^{\epsilon}_{P_i} \big\}$.
Solve \eqref{opt:LP-D} for each $P_i$ ($i=1,2$).
Set $x^{P_i}$ and $\theta^{P_i}$ to be the optimal solution of the linear program associated with $P_i$ ($i=1,2$). 
Set $U[P_i]\gets \theta^{P_i}$ and $L[P_i]\gets f(x^{P_i})$ for $i=1,2$.
Update the global lower bound, upper bound and the current best solution. 

The above procedures are repeated until the difference between the global lower and upper bounds
is smaller than the tolerance. For more details of the G2B2 algorithm, see Algorithm~\ref{alg:gen-GBB}.  
Theorem~\ref{thm:GBB-smooth-finite-convg} shows that the G2B2 algorithm can find an $\epsilon$-optimal 
solution in finitely many iterations. The Numerical Instance~\ref{exp:G2B2} illustrates the first two iterations
of the G2B2 algorithm (Algorithm~\ref{alg:gen-GBB}) for solving a simple numerical instance.

\begin{algorithm}
	\caption{A generalized geometric branch-and-bound algorithm (G2B2) to solve \eqref{opt:FRO-F}. }
	\label{alg:gen-GBB}
	\begin{algorithmic}
	\STATE {\bf Input}: A polytope $X$,  a set of convex candidate functions $\{f_k\}^K_{k=1}$ 
				and the initial approximation $\widehat{f}_k$ (a linear function) of $f_k$ for $k\in[K]$.     
	\STATE {\bf Output}: An $\epsilon$-solution $x^{\epsilon}$ to the problem \eqref{opt:FRO-F}.
	
	\STATE (Initialization) Set the root node $P_0$ as: $P_0\gets X$. 
										Let $f_{P_0}\gets \widehat{f}_k\;\;\forall k\in[K]$,
										$\mK^{\epsilon}_{P_0}\gets\emptyset$, $\mS_{P_0}\gets\emptyset$,
										$L[P_0]\gets -\infty$, $U[P_0]\gets \infty$. Let $L\gets -\infty$ and $U\gets \infty$.
	\WHILE{$|U-L|>\epsilon$}
		\STATE Branching node selection: Pick the leaf node $P$ such that $U[P]=\textrm{max}_{P^{\prime}\in Leaves} U[P^{\prime}]$ 
		\IF{$[K]\setminus \mK^{\epsilon}_P$ is non-empty}	
			\STATE  Choose the $k^*\in [K]\setminus \mK^{\epsilon}_P$ such that 
					 $v^{P,\epsilon}_{k^*}=\textrm{max}\{v^{P,\epsilon}_k:\;k\in [K]\setminus \mK^{\epsilon}_P \}$.	
			\STATE Create two child nodes $P_1$, $P_2$ of $P$ defined in \eqref{opt:master-delta}. Delete $P_i$ ($i=1,2$) if it is empty.								Determine $\big\{ f^{P_i}_k:\; k\in [K] \big\}$ as in \eqref{eqn:def_K_f} for $i=1,2$. 
			\STATE Determine $\mK^{\epsilon}_{P_i}$ and $\mS_{P_i}$ by solving the problem~\eqref{opt:max_fk-ax-b}.
			\STATE Solve \eqref{opt:LP-D} for $P_i$ to obtain $(x^{P_i},\theta^{P_i})$ for $i=1,2$.
			\STATE Set $U[P_i]\gets\theta^{P_i}$ and $L[P^i]\gets f(x^{P_i})$ for $i=1,2$. 
			\STATE Update $L\gets\textrm{max}\{L,\; L[P_1],\; L[P_2] \}$. 
			\STATE Update $U\gets\textrm{max}\{U[P^{\prime}]:\;P^{\prime}\in Leaves \}$
			\STATE Find a leaf $P^*$ of the current GB2 tree such that $L[P^*]=\textrm{max}_{P\in Leaves}\; L[P]$
			             and set $x^{\epsilon}\gets x^{P^*}$.
	       \ELSE
			\STATE Prune the node $P$.
		\ENDIF
		\STATE Pruning: For every $P\in Leaves$, if $U[P]\le L$, then prune the branch lead by the node $P$.
	\ENDWHILE
	
	\STATE Return $x^{\epsilon}$.
	\end{algorithmic}	 	
\end{algorithm}

\begin{theorem}[finite convergence of Algorithm~\ref{alg:gen-GBB}]
\label{thm:GBB-smooth-finite-convg}
Suppose each candidate function $f_k$ in \eqref{opt:FRO-F} is convex and bounded in the compact set $X$.
Algorithm~\ref{alg:gen-GBB} terminates in finitely many iterations and return an $\epsilon$-optimal solution to \eqref{opt:FRO-F}.  
\end{theorem}
\begin{proof}
By the termination criteria, when the algorithm terminates it will return an $\epsilon$-optimal solution to \eqref{opt:FRO-F}.
It suffices to show that the algorithm will terminate in finitely many iterations.
Suppose the algorithm does not terminate in finitely many iterations. Then the algorithm will generate infinitely many
tangent planes for some candidate function $f_n$. Let $S=\{x_i\}^{\infty}_{i=1}$ be the sequence of tangent points
corresponding to the tangent planes. Since all points in $S$ are within the compact set $X$, it contains a convergent
subsequence $\{x_{i_k}\}^{\infty}_{k=1}$. Let $x_0$ be the limit of $\{x_{i_k}\}^{\infty}_{k=1}$.
Let $B(x,r)=\Set*{y}{\|y-x\|_2\le r}$ be the closed ball with the center at $x$. For a tangent point $t$
define the following function $g_t(y)=\textrm{max}_{g\in\partial f_n(t)}|f_n(y)-g^T(y-t)-f_n(t)|$. The value $g_t(y)$ is equal to the gap 
between the function value $f_n$ and the value of tangent plane at the point $y$.  
Since $f_n$ is convex and continuous in $X$, it is also absolutely continuous in $X$ \cite{rudin2013_math-analy}. 
It implies that the function $g_t$ is also absolutely continuous in X. 
Therefore, there exists a constant $\delta>0$ such that $|g_t(x)-g_t(x^{\prime})|<\frac{\epsilon}{4}$ for any $x,x^{\prime}\in X$
satisfying $\|x-x^{\prime}\|_2<\delta$. There exists a sufficient large integer $N$ such that $\|x_{i_k}-x_0\|_2<\delta$ if $k>N$.
It follows that $|g_{x_0}(x_{i_k})-g_{x_0}(x_0)|<\frac{\epsilon}{4}$ if $k>N$.
This implies that the algorithm keeps adding tangent planes to refine the approximation of $f_n$ within the region $B(x_0,\delta)$
even though the approximation error is already less than $\epsilon$, which leads to a contradiction to the mechanism of the algorithm. 
Therefore, the algorithm terminates in finitely many iterations.
\end{proof}

\begin{instance}
\label{exp:G2B2}
\normalfont Consider a (FRO-F) instance with the following two candidate functions defined in $X=[0,10]\times[0,10]$:
\begin{displaymath}
\begin{aligned}
&f_1(x)=4.87x^2_1+2.93x_1x_2+1.25x^2_2-12.67x_1-5.43x_2 + 15.5 \\
&f_2(x)=2x^2_1+4.36x_1x_2+3.2x^2_2 -114.03x_1  -48.87x_2 + 780.81
\end{aligned}
\end{displaymath}
Define the function $f(x)=\textrm{min}\{f_1(x),\;f_2(x)\}$.

\textbf{Iteration 1}: At the beginning of the G2B2 algorithm, each function is approximated by a linear function, which 
is a tangent hyperplane of the function. Suppose $f_1$ are initially approximated by 
the tangent plane at the point $(0,0)$, and $f_2$ is initially approximated by 
the tangent plane at the point $(10,10)$.
This tangent plane function can be written as:
\begin{displaymath}
\begin{aligned}
&g_{11}(x)= -12.67x_1 -5.43x_2 +  15.5, \\
&g_{21}(x)=8.36x_1 + 10.76x_2 -175.19,
\end{aligned}
\end{displaymath}
for $f_1$ and $f_2$, respectively. Let $P_0=X$ be the root node. 
The approximation functions at $P_0$ are $f^{P_0}_1(x)=g_{11}(x)$
and $f^{P_0}_2(x)=g_{21}(x)$. Suppose the tolerance accuracy is set to be
$\epsilon=1\times 10^{-4}$. Since the approximation at $ P_0$ is inaccurate, 
we have $\mK^{\epsilon}_{P_0}=\emptyset$. Solving the following separation problems
\begin{displaymath}
\underset{x\in P_0}{\textrm{max}}\;f_1(x)-g_{11}(x),\qquad   \underset{x\in P_0}{\textrm{max}}\;f_2(x)-g_{21}(x), 
\end{displaymath}
to identify the maximum approximation error for $f_1$ and $f_2$, respectively. 
This can be done by simply enumerate vertices of $ P_0$. The optimal solution and optimal value
of the separation problems are: $x^{P_0}_1=(10,10)$, $v^{P_0}_1=905.0$  and 
$x^{ P_0}_2=(0,0)$, $v^{P_0}_2=956.0$, respectively. We solve the following linear program to 
determine an upper bound associated with $ P_0$:
\begin{displaymath}
\begin{aligned}
&\textrm{max}\;\;\theta \\
&\textrm{ s.t. }\;\;\theta\le f^{ P_0}_1(x)+v^{ P_0}_1 \\
&\qquad\;\;  \theta\le f^{ P_0}_2(x)+v^{ P_0}_2 \\
&\qquad \;\; x\in  P_0.
\end{aligned}
\end{displaymath}
The optimal solution and optimal value of the above LP is $x^{ P_0}=(0,8)$
and $\theta^{P_0}=873.65$, respectively. We set the global upper bound $U=873.65$
and the lower bound $L=f(x^{P_0})=52.06$. 

\textbf{Iteration 2}: We branch the root node $P_0$ by adding a new linear piece to approximate $f_2$.
The new linear piece $g_{22}$ is induced by the tangent plane of $f_2$ at $x^{P_0}_2=(0,0)$. 
The function of $g_{22}$ is given by the following
\begin{equation}
g_{22}(x) = -75.24x_1  -96.84x_2+  780.81.
\end{equation}
The $P_0$ is branched into two nodes $P_1$ and $P_2$, which are defined as:
\begin{equation}
\begin{aligned}
P_1=\Set*{x\in P_0}{g_{21}(x)\ge g_{22}(x)}, \quad P_2=\Set*{x\in P_0}{g_{22}(x)\ge g_{21}(x)}.
\end{aligned}
\end{equation}
The $P_1$ and $P_2$ are shown in Figure~\ref{fig:G2B2-partition}(a).
Solving the separation problem for the two new nodes, we find the following solutions
\begin{displaymath}
\begin{aligned}
&x^{P_1}_1=(10,10),\quad v^{P_1}_1=905,\qquad x^{P_1}_2=(10,1.12),\quad v^{P_1}_2=252.60, \\
&x^{P_2}_1=(10,1.12),\quad v^{P_2}_1=521.23,\qquad x^{P_2}_2=(10,1.12),\quad v^{P_2}_2=252.60.
\end{aligned}
\end{displaymath}
Solving the node relaxation problem for the two new nodes, we find the following solutions:
\begin{displaymath}
\begin{aligned}
&x^{P_1}=(10,10), \qquad \theta^{P_1}=268.61, \\
&x^{P_2}=(0,0),\qquad \theta^{P_2}=536.73.
\end{aligned}
\end{displaymath} 
We update the global upper bound and lower bound as: 
$U=\textrm{max}\{\theta^{P_1},\theta^{P_2}\}=536.73$,
$L=\textrm{max}\{L^{\prime}, f(x^{P_1}), f(x^{P_2})\}=52.06$,
where $L^{\prime}$ is the previous lower bound.

\textbf{Iteration 3}: We branch the node $P_1$ by adding a new linear piece to approximate $f_1$.
The new linear piece $g_{12}$ is induced by the tangent plane of $f_1$ at $x^{P_1}_1$.
The function of $g_{12}$ is given by the following
\begin{equation}
g_{12}(x)=114.03x_1+   48.87x_2 -889.5.
\end{equation}
The $P_1$ is branched into two nodes $P_3$ and $P_4$, which are defined as follows:
\begin{equation}
P_3=\Set*{x\in P_1}{g_{11}(x)\ge g_{12}(x)}, \quad P_4=\Set*{x\in P_1}{g_{12}(x)\ge g_{11}(x)}.
\end{equation}
The $P_2$, $P_3$ and $P_4$ are shown in Figure~\ref{fig:G2B2-partition}(b).
The updated lower and upper bounds can be determined similarly as in Iteration 2.
\end{instance}
\begin{figure}
\centering
\subfigure[]{\includegraphics[width=0.45\textwidth]{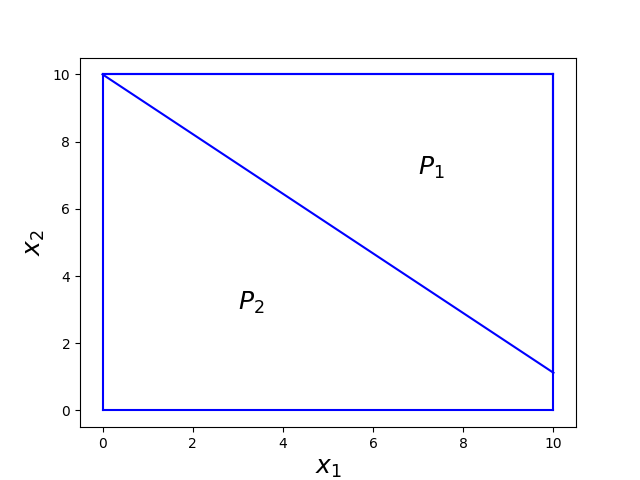}}
\subfigure[]{\includegraphics[width=0.45\textwidth]{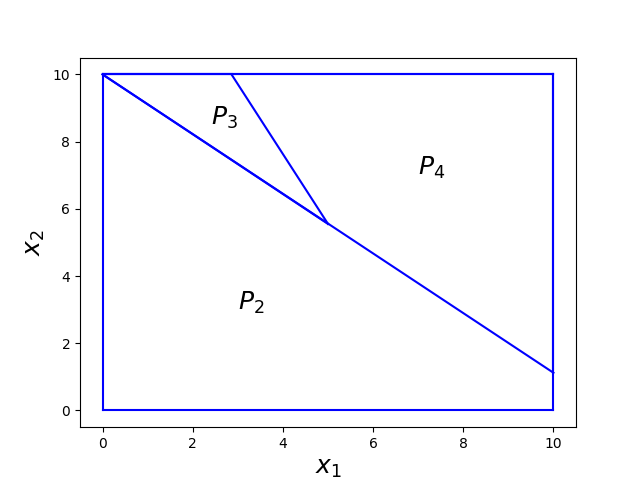}}
\caption{Partition of $X$ into sub-polytopes in the G2B2 algorithm, when Algorithm~\ref{alg:gen-GBB} 
			is applied to Numerical Instance~\ref{exp:G2B2}. }
\label{fig:G2B2-partition}
\end{figure}

\section{Numerical Study}
\label{sec:num-study}
\subsection{Generation of numerical instances}
\label{sec:num-gen-inst}
We generate 29 \eqref{opt:FRO-F} numerical instances to test the numerical performance
of the G2B2 algorithm (Algorithm~\ref{alg:gen-GBB} in Section~\ref{sec:GBB-smooth}).
The size of an instance is determined by the dimension $n$ and the number $K$ of candidate 
functions. The instance ID is given in the first column of Table~\ref{tab:G2B2}, and the two parameters of each instance 
are given in the second column of Table~\ref{tab:G2B2}. For every instance, each candidate function
is set to be a convex quadratic function of the form $f_k(x)=x^{\top}Q^{\top}Qx+b^{\top}x+c$,
where every entry of the matrix $Q$ is drawn from an uniform distribution over the interval $[-3,3]$,
every entry of the vector $b$ and the constant $c$ are drawn from an uniform distribution over the interval $[0,20]$.
The feasible set $X$ is set to be $X=[-10,10]^n$ for every instance.

\subsection{Implementation of  the G2B2 algorithm}
\label{sec:num-impl}
The G2B2 (Algorithm~\ref{alg:gen-GBB}) is implemented using Python 3.6.3 
and all linear programs and mixed 0-1 linear programs in the two algorithms
are solved using Gurobi 7.5.1. 
Since the separation problem is to maximize a quadratic function over a polytope, it is a special
case of constrained polynomial optimization.
We first tried to use three polynomial optimization solvers 
GloptiPoly \cite{gloptipoly}, Ncpol2sdpa \cite{ncpol2sdpa} 
and Polyopt \cite{polyopt} to solve the separation problems.
We find that all the solvers can not give a feasible solution to a simple polynomial
optimization problem that has an analytical optimal solution. Therefore, it is unreliable to use any 
of them to solve our separation problems. 
In addition, we are not able to access a high quality global optimization solver such as BARON 
\cite{baron_software,baron_paper}.
Instead, we propose three heuristic methods for approximately solving the separation problems 
(maximization of convex objective over a polytope) that are needed in the G2B2  algorithm. 
The first method is based on the well known fact that the optimal solution of 
the max-of-convex problem over a polytope locates on one of the extremal points of the polytope. 
We can create a hyper-rectangular that contains the polytope and enumerate all vertices of the
circumscribed hyper-rectangular to get an upper bound of the separation problem. 
Specifically, suppose the separation problem is written as $\textrm{max}_{x\in P}\; g(x)$ the 
circumscribed hyper-rectangular is given by: $[r^1_l,r^1_u]\times\dots\times[r^n_l,r^n_u]$,
where $n$ is the dimension. The values $r^k_l$ and $r^k_u$ are determined by the linear programs 
$\textrm{min}_{x\in P}\;x_k$ and $\textrm{max}_{x\in P}\;x_k$, respectively for every $k\in[n]$.
The second method is to use a local optimization solver to solve the separation problem.
In the implementation, we use the pyipopt 0.7 \cite{pyipopt} (a python API of IPOPT \cite{ipopt}) to find a local optimizer 
of the separation problems. The third method is the same as the second method except that we 
try 100 different initial solution in the polytope and choose the largest objective value when using pyipopt. 
The three heuristic methods for solving the separation problems are summarized in the following list:
\begin{enumerate}
	\item Box: Create a circumscribed hyper-rectangular and enumerate over all vertices of the hyper-rectangular;
	\item LC1: Use a local optimization solver to solve the separation problems;
	\item LC2: Try 100 different initial points when using the local optimization solver.
\end{enumerate}
The G2B2 algorithm can be further particularized into three algorithms G2B2-Box, G2B2-LC1 and G2B2-LC2
according to the chose of the heuristic methods for solving the separation problem. 
For each of the three algorithms we have also added a random-walk procedure before
the main algorithm. The random-walk procedure means to identify a good lower bound of \eqref{opt:FRO-F}
and use it as the initial lower bound in the G2B2 based algorithms.
The random walk is inside the feasible set $X$. It takes a random step size towards a random direction
at each iteration to generate the next candidate point. If the candidate point is feasible, 
it moves to this point, evaluates the function value at this point and updates the 
largest function value theretofore, otherwise it re-generate a candidate point and verify its feasibility. 
The total time spending on the random walk is set to be proportional to the dimension $n$ of the instance
(see the third column in Table~\ref{tab:G2B2}). For every instance, the time of random walk is split into 10
equal portions. At the beginning of each portion, the procedure restarts at the initial point. 
The time limit for solving each instance is 4 hours which includes the time for the random walk. 
In the implementation of the G2B2 based algorithms, all the master and separation problems
are warm started.

\subsection{Numerical performance of the G2B2 based algorithms on solving \eqref{opt:FRO-F} instances}
\label{sec:num-G2B2}
The results of the G2B2-Box, G2B2-LC1 and G2B2-LC2 algorithms are given in Table~\ref{tab:G2B2}.
For the G2B2-based algorithms, the random walk provides a good lower bound in most of the small
and midsize instances (Instances 1-23). In particular, the G2B2-Box, G2B2-LC1 and G2B2-LC2 algorithms
fail to improve the random-walk lower bound in 3 instances (Instances 5, 8, 23), 
in 12 instances (Instances 1, 5, 8-9, 11-12, 14, 16-17, 21-23), 
and in 12 instances (Instances 5-8, 13, 16-17, 19-21, 23, 26), respectively.
It can be seen from Table~\ref{tab:G2B2} that more nodes are created for solving smaller instances,
because the master and separation problems are easier to solve in small instances.
The majority of computational time (more than 90\%) are spent 
on solving the separation problems in the three G2B2-based algorithms. 
The G2B2-Box algorithm can solve 17 instances (Instances 1-17) to optimality. 
It solves Instances 18-20 with the optimality gap in the range 6.5\%-11.5\%,
and it solves Instances 21-26 with the optimality gap in the range 69\%-83.5\%.
The G2B2-Box algorithm fails to provide an optimality gap for the Instances 27-29,
because it is impossible to enumerate all $2^{20}$ vertices of a hyper-rectangular of dimension 20.
For the completely solved instances, the computational time increases with the size of the instance
within the range from 599 seconds to 7178 seconds.
The G2B2-LC1 and G2B2-LC2 algorithms can not provide an optimality gap because they
use a local optimization solver to solve the separation problems and hence no reliable upper bounds
can be identified.  The G2B2-Box algorithm gives a larger objective value than the G2B2-LC1
algorithm in 13 Instances (Instances 1, 9, 11-12, 14, 16-17, 19-22, 24, 26), while the G2B2-LC1
gives a larger objective value in 5 instances (Instances 18, 25, 27-29).
For the rest 11 instances, the two algorithms give the same objective value.
This result indicates that the G2B2-Box algorithm has better performance than the G2B2-LC1 in instances with low and 
midsize dimension. In high-dimensional instances (e.g., Instances 27-29), the G2B2-Box algorithm spends much more time
enumerating vertices of the hyper-rectangular when solving the separation problems, and the nodes created 
of the G2B2-Box algorithm is much less than that of the G2B2-LC1 algorithm in the high-dimensional instances.
The G2B2-Box algorithm gives a larger objective value than the G2B2-LC2
algorithm in 16 Instances (e.g., Instances 4, 6-7, 9, 13-17, 19-21, 24-26, 29), while the G2B2-LC2 algorithm
gives a larger objective value in 3 instances (Instances 22, 27-28).
For the rest 11 instances, the two algorithms give the same objective value.
This indicates that the G2B2-Box algorithm has better performance than the G2B2-LC2 algorithm.
When comparing the performance of G2B2-LC1 and G2B2-LC1 algorithms, it can be seen that although 
the use of multiple-initial-point strategy may find a better solution for the separation problems, 
it fails to improve the objective value because it spends more time on solving the separation problems.

\begin{table}
{\scriptsize
\caption{\scriptsize Numerical results of solving the 29 instances using the G2B2 algorithm.
The columns `RW-T(s)' and `RW-LB' represent the time for doing random walk and the best objective found in
the random walk, respectively. The `NA' means the value is not available. The `*' in the `CPU(s)' column means
that the 4-hour time limit is reached. 
}\label{tab:G2B2}
\begin{tabular}{cc|cc|ccccc}
\hline\hline
\multicolumn{2}{c|}{Instance} & \multicolumn{2}{c|}{Random Walk} & \multicolumn{5}{c}{G2B2-Box} \\
\hline
ID 	&	 dim     K 	&	RW-T(s)	&	RW-LB	&	 Obj 	&	 Gap(\%) 	&	 CPU(s) 	&	 sep-T(\%) 	&	 Nodes 	 \\
\hline
1	&	2   5	&	360	&	809.85	&	815.13	&	0	&	600	&	94.6	&	51182	\\
2	&	2   10	&	360	&	221.56	&	221.58	&	0	&	730	&	97.2	&	15902	\\
3	&	2   15	&	360	&	327.18	&	327.21	&	0	&	1270	&	98.1	&	8289	\\
4	&	2   20	&	360	&	459.52	&	462.54	&	0	&	1895	&	98.5	&	6366	\\
5	&	2   25	&	360	&	243.31	&	243.31	&	0	&	1936	&	98.8	&	4635	\\
6	&	2   30	&	360	&	234.09	&	234.23	&	0	&	2458	&	99	&	3226	\\
7	&	2   35	&	360	&	204.93	&	205.03	&	0	&	3677	&	99.1	&	2179	\\
8	&	2   40	&	360	&	312.39	&	312.39	&	0	&	4058	&	99.1	&	1541	\\
9	&	2   45	&	360	&	145.25	&	148.73	&	0	&	4260	&	99.2	&	1374	\\
10	&	2   50	&	360	&	360.72	&	360.73	&	0	&	4389	&	99.3	&	1708	\\
\hline																	
11	&	3   5	&	540	&	1815.13	&	1888.46	&	0	&	960	&	96.4	&	24462	\\
12	&	3   10	&	540	&	1324.87	&	1327.62	&	0	&	1260	&	98.2	&	8293	\\
13	&	3   15	&	540	&	1028.95	&	1030.11	&	0	&	2463	&	98.7	&	6194	\\
14	&	3   20	&	540	&	789.33	&	813.69	&	0	&	3322	&	99	&	4419	\\
15	&	3   25	&	540	&	777.4	&	782.7	&	0	&	3868	&	99.2	&	2733	\\
16	&	3   30	&	540	&	737.16	&	739.26	&	0	&	4702	&	99.2	&	4234	\\
17	&	3   35	&	540	&	548.25	&	550.31	&	0	&	7178	&	99.3	&	3793	\\
18	&	3   40	&	540	&	570.38	&	570.37	&	6.7	&	*	&	99.3	&	3432	\\
19	&	3   45	&	540	&	474.4	&	488.76	&	11.2	&	*	&	99.3	&	3166	\\
20	&	3   50	&	540	&	636.05	&	650.66	&	7.8	&	*	&	99.4	&	2375	\\
\hline																	
21	&	5   50	&	900	&	2216.17	&	2273.65	&	69.2	&	*	&	99.6	&	2137	\\
22	&	5   100	&	900	&	1244.41	&	1257.9	&	83.2	&	*	&	99.7	&	916	\\
23	&	5   200	&	900	&	1376.56	&	1376.56	&	80.2	&	*	&	99.7	&	516	\\
24	&	10   50	&	1800	&	11558.32	&	14295.6	&	70.8	&	*	&	99.8	&	532	\\
25	&	10   100	&	1800	&	9123.15	&	10718.77	&	79.7	&	*	&	99.8	&	284	\\
26	&	10   200	&	1800	&	7721.63	&	11424.29	&	77.4	&	*	&	99.9	&	146	\\
27	&	20   50	&	3600	&	47634.95	&	60568.18	&	 NA 	&	*	&	99.9	&	122	\\
28	&	20   100	&	3600	&	40574.35	&	57170.26	&	 NA 	&	*	&	99.9	&	60	\\
29	&	20   200	&	3600	&	34514.77	&	52373.71	&	 NA 	&	*	&	99.9	&	34	\\
\hline
\end{tabular}

\vspace*{0.5cm}

 \begin{tabular}{cc|cccc|cccc}
 \hline\hline
\multicolumn{2}{c|}{Instance} &\multicolumn{4}{c|}{G2B2-LC1} & \multicolumn{4}{c}{G2B2-LC2} \\
\hline
ID 	&	 dim     K  	&	 Obj 	&	 CPU(s) 	&	 sep-T(\%) 	&	 Nodes 	&	 Obj 	&	 CPU(s) 	&	 sep-T(\%) 	&	 Nodes 	\\
\hline
1	&	2   5	&	809.85	&	*	&	92.9	&	56058	&	815.13	&	*	&	99.7	&	1543	\\
2	&	2   10	&	221.58	&	*	&	96.3	&	23260	&	221.58	&	*	&	99.8	&	655	\\
3	&	2   15	&	327.21	&	*	&	97.4	&	12856	&	327.21	&	*	&	99.9	&	429	\\
4	&	2   20	&	462.13	&	*	&	98	&	10102	&	460.91	&	*	&	99.9	&	336	\\
5	&	2   25	&	243.31	&	*	&	98.4	&	7552	&	243.31	&	*	&	99.9	&	238	\\
6	&	2   30	&	234.24	&	*	&	98.6	&	5872	&	234.09	&	*	&	99.9	&	231	\\
7	&	2   35	&	205.04	&	*	&	98.9	&	3531	&	204.93	&	*	&	99.9	&	187	\\
8	&	2   40	&	312.39	&	*	&	99	&	2501	&	312.39	&	*	&	99.9	&	95	\\
9	&	2   45	&	145.25	&	*	&	99.1	&	3370	&	145.85	&	*	&	99.9	&	172	\\
10	&	2   50	&	360.73	&	*	&	99.1	&	2949	&	360.73	&	*	&	99.9	&	147	\\
\hline																			
11	&	3   5	&	1815.13	&	*	&	93.4	&	62643	&	1888.45	&	*	&	99.7	&	1228	\\
12	&	3   10	&	1324.87	&	*	&	96.1	&	19680	&	1327.62	&	*	&	99.9	&	623	\\
13	&	3   15	&	1030.11	&	*	&	97.3	&	11868	&	1028.95	&	*	&	99.9	&	457	\\
14	&	3   20	&	789.33	&	*	&	98	&	6940	&	791.31	&	*	&	99.9	&	315	\\
15	&	3   25	&	782.7	&	*	&	98.3	&	5047	&	782.52	&	*	&	99.9	&	226	\\
16	&	3   30	&	737.16	&	*	&	98.5	&	5712	&	737.16	&	*	&	99.9	&	300	\\
17	&	3   35	&	548.25	&	*	&	98.7	&	4742	&	548.25	&	*	&	99.9	&	275	\\
18	&	3   40	&	576.1	&	*	&	98.9	&	3136	&	570.37	&	*	&	99.9	&	188	\\
19	&	3   45	&	476.39	&	*	&	99	&	2826	&	474.4	&	*	&	99.9	&	167	\\
20	&	3   50	&	639.04	&	*	&	99.1	&	1866	&	636.05	&	*	&	99.9	&	129	\\
\hline																			
21	&	5   50	&	2216.17	&	*	&	98.8	&	3012	&	2216.17	&	*	&	99.9	&	126	\\
22	&	5   100	&	1244.41	&	*	&	99.1	&	2028	&	1303.03	&	*	&	99.9	&	66	\\
23	&	5   200	&	1376.56	&	*	&	99.3	&	1016	&	1376.56	&	*	&	99.9	&	34	\\
24	&	10   50	&	12346.47	&	*	&	98.9	&	1894	&	13610.62	&	*	&	99.9	&	72	\\
25	&	10   100	&	11088.59	&	*	&	99.4	&	694	&	10698.17	&	*	&	99.9	&	40	\\
26	&	10   200	&	8228.6	&	*	&	99.5	&	448	&	7721.63	&	*	&	99.9	&	22	\\
27	&	20   50	&	65609.65	&	*	&	98.9	&	1072	&	62231	&	*	&	99.9	&	40	\\
28	&	20   100	&	59613.64	&	*	&	99.4	&	430	&	62468.2	&	*	&	99.9	&	24	\\
29	&	20   200	&	53143.2	&	*	&	99.5	&	288	&	44369.22	&	*	&	99.9	&	12	\\
\hline
 \end{tabular}
}
\end{table}

\section{Concluding Remarks}
The results on solving the numerical instances indicates that the G2B2-Box algorithm
has the best overall performance. Instances with low dimension ($\textrm{dim}=2,3$)
are solved to optimality by the G2B2-Box algorithm. This implies that from the numerical
perspective, it is more practical to apply the functionally-robust optimization framework 
to model problems in low dimensional space.

\section{Acknowledgement}

\bibliographystyle{amsplain}
\bibliography{reference-database-current}

\providecommand{\bysame}{\leavevmode\hbox to3em{\hrulefill}\thinspace}
\providecommand{\MR}{\relax\ifhmode\unskip\space\fi MR }
\providecommand{\MRhref}[2]{%
  \href{http://www.ams.org/mathscinet-getitem?mr=#1}{#2}
}
\providecommand{\href}[2]{#2}
\begin{thebibliography}{10}

\bibitem{adjiman1998-alpha-BB-gen-twice-df-1}
C.~S. Adjiman, S.~Dallwig, C.~A. Floudas, and A.~Neumaier, \emph{A global
  optimization method, $\alpha${BB}, for general twice-differentiable {NLPs-I}.
  theoretical advances}, Compute. Chem. Eng. \textbf{22} (1998), no.~9,
  1137--1158.

\bibitem{adjiman1998-alpha-BB-gen-twice-df-2}
\bysame, \emph{A global optimization method, $\alpha${BB}, for general
  twice-differentiable {NLPs-II}. implementation and computational results},
  Compute. Chem. Eng. \textbf{22} (1998), no.~9, 1159--1179.

\bibitem{book_robust_opt2009}
A.~Ben-Tal, L.~EI Ghaoui, and A.~Nemirovski, \emph{Robust optimization},
  Princeton Series in Applied Mathematics, Princeton University Press, August
  2009.

\bibitem{Ben-Tal2000_robustness}
A.~Ben-Tal, L.~EL Ghaoui, and A.~Nemirovski, \emph{Robustness}, Handbook of
  Semidefinite Programming (Dordrecht, Netherlands) (R.~Saigal,
  L.~Vandenberghe, and H.~Wolkowicz, eds.), Kluwer Academic, 2000,
  pp.~139--162.

\bibitem{Ben-Tal-1998_rob-conv-opt}
A.~Ben-Tal and A.~Nemirovski, \emph{Robust convex optimization}, Mathematics of
  Operations Research \textbf{23} (1998), 769--805.

\bibitem{Ben-Tal-1999_rob-lin-prog}
\bysame, \emph{Robust solutions to uncertain linear programs}, Operations
  Research Letters \textbf{25} (1999), 1--13.

\bibitem{benson2006-frac-prog-conv-quad-form-func}
H.~P. Benson, \emph{Fractional programming with convex quadratic forms and
  functions}, European Journal of Operational Research \textbf{173} (2006),
  351--369.

\bibitem{benson2007-BB-dual-bd-alg-sum-lin-ratio}
\bysame, \emph{A simplicial branch and bound duality-bounds algorithm for the
  linear sum-of-ratios problem}, Eur. J. Oper. Res. \textbf{182} (2007), no.~2,
  597--611.

\bibitem{benson2007-sum-ratio-frac-prog-concav-min}
\bysame, \emph{Solving sum of ratios fractional programs via concave
  minimization}, Journal of Optimization Theory and Applications \textbf{135}
  (2007), 1--17.

\bibitem{bertsimas2005_opt-ineq-prob}
D.~Bertsimas and I.~Popescu, \emph{Optimal inequalities in probability theory:
  A convex optimization approach}, SIAM J. OPTM. \textbf{15} (2005), 780--804.

\bibitem{bonami2013-branch-rule-conv-MINLP}
P.~Bonami, J.~Lee, S.~Leyffer, and A.~W{\"a}chter, \emph{Novel bound
  contraction procedure for global optimization of bilinear minlp problems with
  applications to water management problems}, Computers and Chemical
  Engineering \textbf{35} (2011), 446--455.

\bibitem{chen2017_spat-BC-nonconv-QCQP-compl-vars}
C.~Chen, A.~Atamt{\"u}rk, and S.~S. Oren, \emph{A spatial branch-and-cut method
  for nonconvex qcqp with bounded complex variables}, Mathematical Programming
  \textbf{165} (2017), 549--577.

\bibitem{conforti-IP-2014}
M.~Conforti, G.~Cornu{\'e}jols, and G.~Zambelli, \emph{Integer programming},
  Graduate Texts in Mathematics, Springer, 11 2014.

\bibitem{faria2011-novel-bound-contract-proc-bilin-MINLP-appl-water-manag}
D.~C. Faria and M.~J. Bagajewicz, \emph{Novel bound contraction procedure for
  global optimization of bilinear minlp problems with applications to water
  management problems}, Computers and Chemical Engineering \textbf{35} (2011),
  446--455.

\bibitem{floudas1999}
C.~A. Floudas, \emph{Deterministic global optimization: theory, algorithms and
  applications}, 1 ed., Springer U.S., 2000.

\bibitem{gerard2017_dive-for-spat-BB}
D.~Gerard, M.~K{\"o}ppe, and Q.~Louveaux, \emph{Guided dive for the spatial
  branch-and-bound}, Journal of Global Optimization \textbf{68} (2017),
  685--711.

\bibitem{gloptipoly}
D.~Henrion and J.~B. Lasserre, \emph{Gloptipoly: global optimization over
  polynomials with {MATLAB} and {SeDuMi}}, ACM Transactions on Mathematical
  Software \textbf{29} (2003), 165--194.

\bibitem{hettich1993_SIP-thy-methd-appl}
R.~Hettich and K.~O. Kortanek, \emph{Semi-infinite programming: theory,
  methods, and applications}, SIAM REVIEW \textbf{35} (1993), no.~3, 380--429.

\bibitem{jiao2006-glob-opt-gen-lin-frac-prog-with-NL-constr}
H.~Jiao, Y.~Guo, and P.~Shen, \emph{Global optimization of generalized linear
  fractional programming with nonlinear constraints}, Applied Mathematics and
  Computation \textbf{183} (2006), 717--728.

\bibitem{kirst2015_upper-bd-spat-BB}
P.~Kirst, O.~Stein, and P.~Steuermann, \emph{Deterministic upper bounds for
  spatial branch-and-bound methods in global minimization with nonconvex
  constraints}, TOP \textbf{23} (2015), 591--616.

\bibitem{leyffer2001-integ-SQP-BB-MINLP}
S.~Leyffer, \emph{Integrating sqp and branch-and-bound for mixed integer
  nonlinear programming}, Comput. Optim. Appl. \textbf{14} (2001), 295--309.

\bibitem{linderoth2005-simp-BB-alg-quad-prog}
J.~Linderoth, \emph{A simplicial branch-and-bound algorithm for solving
  quadratically constrained quadratic programs}, Math. Program. Ser. B
  \textbf{103} (2005), 251--282.

\bibitem{luo2019-PL-rob-opt}
F.~Q. Luo and S.~Mehrotra, \emph{Robust maximization of piecewise-linear convex
  functions using mixed binary linear programming reformulation}, Tech. report,
  Northwestern University, Department of Industrial Engineering and Management
  Science, 2019.

\bibitem{pozo2010_spat-BB-opt-metabolic-netwk}
C.~Pozo, G.~Guill{\'e}n-Gos{\'a}lbez, A.~Sorribas, and L.~Jim{\'e}nez, \emph{A
  spatial branch-and-bound framework for the global optimization of kinetic
  models of metabolic networks}, Industrial \& Engineering Chemistry Research
  \textbf{50} (2010), no.~9, 5225--5238.

\bibitem{pronzato2012_space-filling-des}
L.~Pronzato and W.~G. M{\"u}ller, \emph{Design of computer experiments: space
  filling and beyond}, Stat. Comput. \textbf{22} (2012), no.~3, 681--701.

\bibitem{rudin2013_math-analy}
W.~Rudin, \emph{Principles of mathematical analysis}, Example Product
  Manufacturer, Jan 2013.

\bibitem{ryoo1995-glob-opt-NLP-MINLP-proc-design}
H.~S. Ryoo and N.~V. Sahinidis, \emph{Global optimization of non convex {NLPs}
  and {MINLPs} with applications in process design}, Computers \& Chemical
  Engineering \textbf{19} (1995), 551--566.

\bibitem{ryoo1996-branch-reduce-glob-opt}
\bysame, \emph{A branch-and-reduce approach to global optimization}, Journal of
  Global Optimization \textbf{8} (1996), 107--138.

\bibitem{baron_software}
N.~V. Sahinidis, \emph{{BARON 17.8.9: Global Optimization of Mixed-Integer
  Nonlinear Programs, {\em User's Manual}}}, 2017.

\bibitem{sahinidis1998-finite-alg-glob-min-sep-concav-prog}
J.~P. Shectman and N.~V. Sahinidis, \emph{A finite algorithm for global
  minimization of separable concave programs}, Journal of Global Optimization
  \textbf{12} (1998), 1--36.

\bibitem{shen2007-glob-opt-sum-gen-ploy-frac-func}
P.~P. Shen and G.~X. Yuan, \emph{Global optimization for the sum of generalized
  polynomial fractional functions}, Math. Methods Oper. Res. \textbf{65}
  (2007), no.~3, 445--459.

\bibitem{sherali1998-glob-opt-nonconv-poly-prog-ratio}
H.~D. Sherali, \emph{Global optimization of nonconvex polynomial programming
  problems having rational exponents}, J. Glob. Optim. \textbf{12} (1998),
  no.~3, 267--283.

\bibitem{sherali2001-glob-opt-nonconv-fact-prog}
H.~D. Sherali and H.~Wang, \emph{Global optimization of non convex factorable
  programming problems}, Math. Program. Ser. A \textbf{89} (2001), 459--478.

\bibitem{smith1999_spat-BB-nonconv-MINLP}
E.~M.~B Smith and C.~C. Pantelides, \emph{A symbolic reformulation/spatial
  branch-and-bound algorithm for the global optimization of nonconvex minlps},
  Computers and Chemical Engineering \textbf{23} (1999), 457--478.

\bibitem{smith1996-glob-opt-gen-proc-model}
E.M.B. Smith and C.~C. Pantelides, \emph{Global optimization of general process
  models}, Global optimization in engineering design (I.~E. Grossmann, ed.),
  Kluwer Academic, Boston, MA, 1996, pp.~355--386.

\bibitem{stein2013_enhan-spat-BB-glob-opt}
O.~Stein, P.~Kirst, and P.~Steuermann, \emph{An enhanced spatial
  branch-and-bound method in global optimization with nonconvex constraints},
  2013, \url{http://www.optimization-online.org/DB_FILE//2013/04/3810.pdf}.

\bibitem{tawarmalani-thesis2001}
M.~Tawarmalani, \emph{Mixed integer nonlinear programs: theory, algorithms and
  applications}, Ph.D. thesis, University of Illinois, Urbana-Champaign, IL,
  2001.

\bibitem{tawarmalani2004_glob-opt-MINLP-they-comp}
M.~Tawarmalani and N.~V. Sahinidis, \emph{Global optimization of mixed-integer
  nonlinear programs: a theoretical and computational study}, Math. Program.
  Ser. A \textbf{99} (2004), 563--591.

\bibitem{baron_paper}
M.~Tawarmalani and N.~V. Sahinidis, \emph{{A polyhedral branch-and-cut approach
  to global optimization}}, Mathematical Programming \textbf{103} (2005),
  225--249.

\bibitem{thakur1991-dom-contr-NLP}
L.~Thakur, \emph{Domain contraction in nonlinear programming: minimizing a
  quadratic concave objective over a polyhedron}, Mathematics of Operations
  Research \textbf{16} (1991), 390--407.

\bibitem{polyopt}
P.~Trutman, \emph{Polynomial optimization problem solver},
  \url{https://github.com/PavelTrutman/polyopt}, 2017.

\bibitem{ipopt}
A.~W{\"a}chter and L.~T. Biegler, \emph{On the implementation of a primal-dual
  interior point filter line search algorithm for large-scale nonlinear
  programming}, Mathematical Programming \textbf{106} (2006), no.~1, 25--57.

\bibitem{wang2005-BB-alg-glob-sum-lin-ratios}
Y.~Wang, P.~Shen, and L.~Zhian, \emph{A branch-and-bound algorithm to globally
  solve the sum of several linear ratios}, Applied Mathematics and Computation
  \textbf{168} (2005), 89--101.

\bibitem{wang2005-glob-opt-gen-geom-prog}
Y.~J. Wang and Z.~Liang, \emph{A deterministic global optimization algorithm
  for generalized geometric programming}, Appl. Math. Comput. \textbf{168}
  (2005), 722--737.

\bibitem{ncpol2sdpa}
P.~Wittek, \emph{Algorithm 950: {Ncpol2sdpa}--sparse semidefinite programming
  relaxations for polynomial optimization problems of non-commuting variables},
  ACM Transactions on Mathematical Software \textbf{41} (2003), no.~21.

\bibitem{pyipopt}
E.~Xu, \emph{A {Python} connector to {IPOPT}},
  \url{https://github.com/xuy/pyipopt}, 2014--2018.

\bibitem{zamora1998-cont-glob-opt-struct-proc-sys-mod}
J.~M. Zamora and I.~E. Grossmann, \emph{Continuous global optimization of
  structured process systems models}, Computers and Chemical Engineering
  \textbf{22} (1998), no.~12, 1749--1770.

\bibitem{zamora1999-branch-contract-alg-concave-univ-bilin-lin-frac-term}
\bysame, \emph{A branch and contract algorithm for problems with concave
  univariate, bilinear and linear fractional terms}, Journal of Global
  Optimization \textbf{14} (1999), 217--249.

\bibitem{zou2004_sensor-deploy}
Y.~Zou and K.~Chakrabarty, \emph{Sensor deployment and target localization in
  distributed sensor networks}, ACM Transactions on Embedded Computing Systems
  \textbf{3} (2004), no.~1, 61--91.

\end{thebibliography}

\end{document}